\documentclass[11pt,a4paper]{amsart}
\usepackage{graphicx}
\usepackage{amssymb}
\usepackage[all]{xy}
\usepackage{amsmath}
\usepackage[french,english]{babel}
\usepackage{amssymb,amsmath,amsfonts,amsthm,amscd}
\usepackage{indentfirst,graphicx}
\usepackage[font={scriptsize},captionskip=5pt]{subfig}

\def\P2{{\mathbb{P}^{2}}}
\def\P3{{\mathbb{P}^{3}}}
\def\P5{{\mathbb{P}^{5}}}
\def\P8{{\mathbb{P}^{8}}}

\def\Sh{\hom_s(\mathbb{C}^n,\mathbb{C}^n)}

\def\C{{\mathbb{C}}}

\def\R{{\mathbb{R}}}

\def\X{{\mathcal{X}}}
\def\O{{\mathcal{O}}}

\def\R{{\mathcal{R}}}
\def\P{{\mathbb{P}}}
\def\B{{\mathcal{B}}}

\DeclareMathOperator{\Projan}{Projan}
\DeclareMathOperator{\mult}{mult}
\DeclareMathOperator{\dist}{dist}
\DeclareMathOperator{\colength}{colength}

\newtheorem{theorem}{Theorem}[section]
\newtheorem{lemma}[theorem]{Lemma}     
\newtheorem{corollary}[theorem]{Corollary}
\newtheorem{definition}[theorem]{Definition}
\newtheorem{proposition}[theorem]{Proposition}

\newtheorem{remark}[theorem]{Remark}

\begin{document}

\title [Symmetric Determinantal Singularities]{Symmetric Determinantal Singularities II: Equisingularity and SEIDS}

\author[T. Gaffney]{Terence Gaffney}\thanks{T.~Gaffney was partially supported by PVE-CNPq Proc. 401565/2014-9}
 \address{T. Gaffney, Department of Mathematics\\
  Northeastern University\\
  Boston, MA 02215}

\author[M. Molino]{Michelle Molino}
\address {M. Molino, Department of Mathematics\\
  Universidade Federal Fluminense\\
  Niteroi, RJ - Brazil}

\begin{abstract}
This paper is the second part of a two part paper which introduces the study of the Whitney Equisingularity of families of Symmetric determinantal singularities. This study reveals how to use the multiplicity of polar curves associated to a generic deformation of a singularity to control the Whitney equisingularity type of these curves.
\end{abstract}

\maketitle

\selectlanguage{english}

\section*{Introduction}

In this paper and part I \cite{P1}, we study the Whitney equisingularity of families of symmetric determinantal varieties. It is part of a long term effort by several researchers to connect invariants of algebraic objects (rings, ideals and modules) associated with singularities of complex spaces to equisingularity conditions. The project took off with work of Bernard Teissier in the 70s. Teissier, in \cite{CESPCW}, in the case of families of hypersurfaces, with isolated singularities, found integral closure descriptions of equisingularity conditions, Whitney A and B, and controlled these conditions using algebraic invariants, such as multiplicities of ideals. Gaffney, in a series of papers, \cite {MEIG}, \cite{SNHS}, \cite {SIDM}, extended the results of Teissier to families of complete intersection, isolated singularities, hypersurfaces with non-isolated singularities, and then constructed a framework for dealing with isolated singularities in general in \cite{PMIPME}. The approach of \cite{PMIPME} is based on pairs of modules $M,N$, $M\subset N$. The choice of $M$ is canonical: it is the module generated by the partial derivatives of the defining equations of $X$, known as the Jacobian module. The best choice of $N$ was less well understood, as some obvious choices lead to technical difficulties in calculating invariants associated to $N$. Recent work by Gaffney and Rangachev (\cite {Gaff1}) supports the following approach adapted to the symmetric case:

In studying the equisingularity of a family of isolated singularities, fix in advance a component of the base space of the versal deformation, from which the given family is induced. This can be done explicitly in the case of determinantal singularities by fixing the size of the presentation matrix. This has the effect of fixing the generic fiber of the versal deformation to which all members of our family can be deformed. Invariants associated with the geometry/topology of this general member provide important information about the singularities in the original family. The infinitesimal deformations of $X$ induced from the fixed component then give the elements of $N$. This process of fixing a class of deformations, and thereby fixing a set of generic elements and a set of infinitesimal deformations, we call fixing the landscape of $X$. 

Determinantal singularities include complete intersection singularities and are the natural next class to study beyond them. After studying such a class of singularities one can ask if the machinery of Gaffney-Rangachev in \cite{Gaff1} extends to other classes of matrices? This paper shows the answer is yes. Given $F$, a $n\times n$ symmetric matrix, with entries in $\O_q$, we can view $F$ as a map from $\C^q\to\Sh$. Fix a set of linear maps  $S_r$ in $\Sh$ of rank less than or equal to $r$, and define $X$ as $F^{-1}(S_r)$. 

Through $F$, we can view symmetric determinantal singularities as sections of the rank singularities of a generic matrix. For the symmetric determinantal singularities, the module $N$ thus obtained is universal; that is, the $N$ for a section, is just the pull-back of $N$ for a rank singularity of the generic matrix. Further, $N$ for the rank singularities is stable; that is it agrees with the Jacobian module of the generic rank singularity. This means that the invariants associated with $N$ for $X$ can be calculated as the intersection number of the image of $F$ with the polar varieties of complementary dimension of  $S_r$. 

This approach extends to any class of sets obtained as sections of rigid singularities, and the invariants so obtained control the Whitney equisingularity of the class of deformations so defined. From these  invariants the natural generalizations of the multiplicity invariants of Teissier can be obtained.

In the first section we recall the definition for some equisingularities conditions, such as Whitney's conditions $A$ and $B$ and Verdier's condition $W$. The Whitney stratification is stated in this section together with some details about the Multiplicity Polar Theory and some ways to calculate the multiplicity of a pair of ideals or modules. 

In the second section we defined the symmetric essentially isolated singularities, SEIDS. The definition of a SEIDS was inspired on the work of Ebeling and Gusein -Zade in \cite{EGZ}.  We also prove the first results on SEIDS based on the work of Gaffey and Ruas in \cite{EAE} for EIDS.

Finally,  in the third section In Section $3$ we use the results from the last sections to provide criteria for Whitney equisingularity conditions when $\X$ is a family of Symmetric Essentially Isolated Determinantal Singularities. 

\section{Multiplicities and Whitney Equisingularity}\label{MWE}

Equisingularity is the study of the behavior of  families of sets or maps with respect to an equivalence relation. The goal is to study this behavior in order to understand the relations between the members of the family, for example, when are all the members the same or how do we tell when a family of sets are the same using invariants of the members of the family? There are many ways to see the members of a family are the same. In the complex analytic case we can use the Whitney's conditions or Verdier's W, known to be equivalent in the complex analytic case, to say when the members of a family are the same. This conditions imply all of the above possible answers. The theory of integral closure of ideals and modules provides an algebraic description of these conditions from which we may abstract the invariants which control them in families.

Let $X$ be a analytic set, $X_0$ the set of smooth points on $X$ and $Y$ a smooth subset of $X$. The pair $(X_0,Y)$ satisfies \textbf{Whitney's Condition A} at $y\in Y$ if for all sequences $\{x_i\}$ of points of $X_0$ and tangent spaces $T_{x_i}X$,
$$
\begin{array}{rcl}
\{x_i\} & \rightarrow & y\\
\{T_{x_i}X\} & \rightarrow & T
\end{array}\Rightarrow T_yY\subset T.
$$

Looking at the same set-up as in the previously case, we say that a pair $(X_0,Y)$ satisfies \textbf{Whitney's Condition B} at $y\in Y$ if for all sequences $\{x_i\}$ of points of $X_0$ and $\{y_i\}$ of points of $Y$, tangent planes $T_{x_i}X$ and secant lines $(x_i,y_i)$,
$$
\begin{array}{rcl}
\{x_i\} & \rightarrow & y\\
\{y_i\} & \rightarrow & y\\
\{T_{x_i}X\} & \rightarrow & T\\
\sec(x_i,y_i) & \rightarrow & L
\end{array}\Rightarrow L\subset T.
$$

\begin{proposition} If a pair $(X_0,Y)$ satisfies Whitney's condition B at \linebreak $y\in Y$, then it also satisfies Whitney's condition A.
\end{proposition}
\begin{proof}
Cf. proposition 2.4 in \cite{mather}.
\end{proof}

Before we state the condition $W$, it is important to define the distance between two liner spaces. Suppose $A,B$ are linear subspaces at the origin in $\C^n$. Then, the distance from $A$ to $B$ is:
$$
\dist(A,B)=\sup\frac{|\left\langle u,v\right\rangle|}{||u||||v||}
$$
where $u\in B^{\perp}-\{0\}$ and $v\in A-\{0\}$. In the applications $B$ is the "big" space and $A$ is the "small" one. The inner product is the Hermitian inner product when we work over $\C$. The same formula works over $\mathbb{R}$. 

\begin{definition} Suppose $Y\subset\overline{X}$, where $X,Y$ are strata in a stratification of a analytic space, and 
$$
\dist(T_yY,T_xX)\leq C\dist(x,Y)
$$
for all $x$ close to $Y$. Then, the pair $(X,Y)$ satisfies \textbf{Verdier's condition W} at $y\in Y$.  
\end{definition}

Conditions W says that the distance between the tangent space $T_{x_i}X$, where $x_i$ is a point of $X_0$, and the tangent space $T_yY$ goes to zero as fast as the distance between $x_i$ and $Y$.

\begin{remark} In \cite{MPSPCW}, Tessier proved (chapter $V$, thm $1.2$) that, for the complex analytic spaces, the condition $W$ is equivalent to condition $A$ and $B$. So, in this paper,  we will use these two terms interchangeably.
\end{remark}

As a first step to understanding the condition, we consider the case where $X$ is a hypersurface in $\C^n$. We would like to re-write this condition in terms of a map $F$ that defines $X$. This will allow us to develop an algebraic formulation of the $W$ condition.

\begin{theorem}\label{WC} Let $\X$ be a family of hypersurfaces in $Y\times\C^{n+1}$. Condition $W$ holds for $(\X_0,Y)$ at $0\in Y$ if, and only if, there exists a neighborhood $U$ of $(0,0)$ in $\X$ and $C>0$ such that 
$$
\displaystyle{\left|\left|\frac{\partial F}{\partial y_l}(y,z)\right|\right|\leq C\sup_{i,j}\left|\left| z_i\frac{\partial F}{\partial z_j}(y,z)\right|\right|},
$$
for all $(y,z)\in U$ and for $1\leq l\leq k$.
\end{theorem}
\begin{proof}
See \cite{SNTG}, proposition $1.8$.
\end{proof}

If $m_Y$ is the ideal defining $Y$, given by $m_Y=(z_1,\ldots,z_n)$, then the inequality above says that the partial derivatives of $F$ with respect to $y_l$ go to zero as fast as the ideal $m_yJM_z(\X)$. We will examine the implications of this fact in further sections.

\begin{definition}
Let $X=\bigcup X_{\alpha}$ be a stratification of $X$. Then, a \textbf{Whitney stratification} of $X$ is a stratification such that for any pair of strata $X_{\beta},X_{\alpha}$ with $X_{\alpha}\subset\overline{X_{\beta}}$, the pair $\left(\overline{X_{\beta}},X_{\alpha}\right)$ satisfies the Whitney conditions at every point $x\in X_{\alpha}$.  
\end{definition}

Whitney showed that complex analytic sets have Whitney stratifications and that these stratifications are preserved by analytic equivalence. This means that if a stratification is locally analytically trivial then it is a Whitney stratification.

Moving forward, the multiplicity of an ideal, module or pair of modules is one of the most important invariants we can associate to an $\mu$-primary module. It is intimately connected with integral closure and it has both a length theoretic definition and intersection theoretic definition.  For a pair of modules $M\subset N$ such that $M$ has finite colength in $N$ we can define the Buchsbaum-Rim multiplicity, denoted $e(M,N)$, as follows:

Let $(X,x)$ be a germ of a complex analytic space, $X$ a small representative of the germ, and  let $\O_X$ denote its structure sheaf. For simplicity we assume $X$ equidimensional. Take $M$ a subsheaf of the free sheaf $\O^p_X$, $g$ the generic rank of $M$ on each component of $X$, and suppose $M$ is also a subsheaf of $N$, which also has generic rank $g$. At a small neighborhood of any point $x$ of $X$ we can consider $M$ and $N$ as submodules of $\O^p_{X,x}$. The multiplicity of a pair of modules $M,N$ is:
$$
e(M,N)=\sum_{j=0}^{d+g-2} \int D_{M,N}\cdot C_M^{d+g-2-j}\cdot C_N^{j}.
$$ 
where $C_M$ and $C_N$ are Chern classes of tautological bundles and $D_{M,N}$ is the exceptional divisor.

Now, suppose we have a family of modules $M\subset N$, $M$ and $N$ submodules of a free module $F$ of rank $p$ on a equidimensional family of spaces with equidimensional fibers $\X^{d+k}$, $\X$ a family over a smooth base $Y^k$. We assume that the generic rank of $M$ and $N$ is $g\leq p$ on every component of the fibers. 

We will be interested in computing, as we move from the special point $0$ to a generic point, the change in the multiplicity of a pair $(M,N)$, denoted by $\Delta e(M,N)$. Let us assume that the integral closures of $M$ and $N$ agree off a set $C$ of dimension $k$ which is finite over $Y$, and assume that we are working on a sufficiently small neighborhood of the origin, so that every component $C$ contains the origin in its closure. Then $e(M,N,y)$ is the sum of the multiplicity of a pair at all points in the fiber of $C$ over $y$, and $\Delta e(M,N)$ is the change in this number from $0$ to a generic value of $y$. If we have a set $S$ which is finite over $Y$, then we can project $S$ to $Y$, and the degree of the branched cover at $0$ is $\mult_yS$, that is, the number of points in the fiber of $S$ over our generic $y$. Let $C(M)$ denote the locus of points where $M$ is not free, \textit{i.e.}, the points where the rank of $M$ is less than $g$, $C(\Projan\R(M))$ its inverse image under $\pi_M$.

\begin{theorem} Suppose in the above setup we have that $\overline{M}=\overline{N}$ off a set $C$ of dimension $k$ which is finite over $Y$. Suppose further that \linebreak $C(\Projan\R(M))(0)=C(\Projan\R(M(0)))$ except possibly at the points \linebreak which project to $0\in\X(0)$. Then, for $y$ a generic point of $Y$,
$$
\Delta(e(M,N))=\mult_y\Gamma_d(M)-\mult_y\Gamma_d(N)
$$
where $\X(0)$ is the fiber over $0$ of the family $\X^{d+k}$, $C(\Projan\R(M))(0)$ is the fiber of $C(\Projan\R(M))$ over $0$ and $M(0)$ is the restriction of the module $M$ to $\X(0)$.
\end{theorem}
\begin{proof}
The proof in the ideal case appears in \cite{PMIPME}; the general proof will appear in \cite{MPFE}.
\end{proof}

\section{First Results on SEIDS}

The essentially Isolated Determinantal Singularities, also known as EIDS, were defined by Ebeling and Gusein-Zade. In their definition, a germ $(X,0)$ contained in $\C^q$ of a determinantal variety of type $(n+k,n,t)$, see \cite{EGZ} for details, is an EIDS if it has only essentially isolated singular points in a punctured neighborhood of the origin in $X$. In this section we are going to give similar definitions for the symmetric case and prove some results about equisingularity in this case.

For this, let $F$ be a map from $\C^q$ to $\Sh$ whose entries are complex analytic functions such that $f_{ij}(x)=f_{ji}(x)$, and suppose, without loss of generallity that $F(0)=0$. Thus, $X$ is a symmetric determinantal variety given by $X=F^{-1}(S_r)$. As we have seen, the representation of $S_r$ as the union 
$$
S_r=\bigcup_{i=0}^r S_i\backslash S_{i-1}
$$  
is a Whitney stratification of $S_r$. If the map $F$ is transverse to the rank stratification of $\Sh$ for $x\neq 0$, then the germ of $X$ at a singular point is holomorphic to either the product of $S_r$ with a affine space or a transverse slice of $S_r$. This inspires the following definition:

\begin{definition} A point $x\in X=F^{-1}(S_r) $ is called symmetric essentially non singular if, at the point $x$, the map $F$ is transversal to the corresponding stratum of $S_r$.
\end{definition}

A germ $(X,0)\subset(\C^q,0)$ of a symmetric determinantal variety $S_r$ is a \textit{symmetric essentially isolated determinantal singularity}, SEIDS, if it has only essentially non-singular points in a punctured neighborhood of the origin in $X$, that is, $F$ is transverse to all strata $S_i\backslash S_{i-1}$ of the stratification of $S_r$ in a punctured neighborhood of the origin. The singular set of a SEIDS $X=F^{-1}(S_r)$ is the SEIDS $F^{-1}(S_{r-1})$.

\begin{definition} A  deformation $\widetilde{F}: U \subset  \C^q   \to \Sh$ of $F$ which is transverse to the rank stratification is called a symmetric stabilization of $F$. The variety $\X=\widetilde{F}^{-1}(S_r)$ is an symmetric essential smoothing of $X$ (\cite[Section 1]{EGZ}). If $\X$ is smooth, then $\X$ is a symmetric  smoothing of $X$.
\end{definition}

In this paper, unless stated otherwise, all smoothings and stabilizations are symmetric smoothings and stabilizations.

\begin{proposition} A SEIDS $(X,0)\subset\C^q$ has an isolated singularity at the origin if and only if
$$
q\leq\frac{(n-r+1)(n-r+2)}{2}.
$$ 
\end{proposition}
\begin{proof}
The singular set of $X$ is the variety $F^{-1}(S_{r-1})$ of codimension \linebreak $\frac{(n-r+1)(n-r+2)}{2}$. Therefore, the origin is an isolated singularity if and only if  $F^{-1}(S_{r-1})$ has dimension less than equal to $0$, which means
$$
\begin{array}{rcl}
\dim F^{-1}(S_{r-1})\leq 0 & \Leftrightarrow & q-\frac{(n-r+1)(n-r+2)}{2}\leq 0\\
&&\\
&\Leftrightarrow& q\leq\frac{(n-r+1)(n-r+2)}{2}.
\end{array}
$$
\end{proof}

In the sequel we shall consider deformations of symmetric determinantal singularities which are themselves symmetric determinantal ones, i.e., they are defined by perturbations of the map $F$ defining the singularity. In \cite{P1}, for the curve and surface cases we considered $\X$ a smoothing of $X$. However, it is not true that we always can get a smoothing from a SEIDS. The next lemma shows the conditions needed to have such deformation. 

\begin{proposition} $(X,0)$ has a smoothing if and only if 
$$
q<\frac{(n-r+1)(n-r+2)}{2}.
$$ 
\end{proposition}
\begin{proof}
If $\X$ is a smoothing of $X$ then all the fibers $\X(y)$, $y\neq 0$, are smooth, which means that $\widetilde{f}_y^{-1}(S_{r-1})$, where $\widetilde{f}_y$ is the restriction to the fiber $\X(y)$, is an empty set. Therefore, $\X$ is a smoothing if and only if
$$
\dim\widetilde{f}_y^{-1}(S_{r-1})<0\Leftrightarrow q<\frac{(n-r+1)(n-r+2)}{2}.
$$
\end{proof}

In the SEIDS context, instead of a smoothing, we are going to consider $\X$ as a stabilization, which means, $\X$ is a symmetric determinantal deformation of $X$ to the generic fiber. 

Let $Y$ be the parameter space which is embedded in $\X$ as a linear subspace. Let us also consider $_iX$ as the preimage $F^{-1}(S_i)$ for all $i=1,\ldots,r$ and, consequently, $_i\X=\widetilde{F}^{-1}(S_i)$. For simplicity of exposition we impose a condition on our families that will be next stated. 

\begin{definition} A $s$-parameter family $\X$ of SEIDS is good if there exists a neighborhood $U$ of $Y$ such that the restriction map $F_{\X(y)}$ to the fiber $\X(y)$ is transverse to the rank stratification off the origin for all $y\in U$.
\end{definition}

\begin{proposition}\label{good} Suppose $\X$ is a good family of SEIDS. Then $\{_i\X\backslash Y\}$, $i=1,\ldots,t$ is a locally analytically trivial stratification of $\X\backslash Y$.
\end{proposition}
\begin{proof} The proof is similar to the rectangular case presented in \cite{EAE}. From the definition of good family for $i$, $\{_i\X\backslash Y\}$ is smooth, since $F$ is transverse to the rank stratification. That the strata are locally analytically trivial follows because the strata by the $S_i$ are and $F$ is transverse to these. 
\end{proof}

This condition implies that in a $s$-parameter family, no set of points from $_i\X$ can split off from the origin if the expected dimension of $_i\X$ is less than $s$. It also implies that there cannot be subsets of points where $\{_i\X\backslash Y\}$ is singular which include the origin in their closure even if the expected dimension of $_i\X$ is greater than or equal to $s$.

One of the invariants we are interested in is the multiplicity of the pair presented in section \ref{MWE}. First, we need to show that such multiplicity \linebreak $e(JM(X),N(X))$ is well-defined for SEIDS. We will proceed as Gaffney and Ruas in \cite{EAE} for the EIDS case, that is, we will show that the transversality of $F$ to the strata of $\Sh$ off the origin implies that off the origin the two modules agree. To show that the multiplicity is well defined, we will use the fact that our stratification is Whitney.

\begin{lemma} Let $F:(\C^q,0)\longrightarrow (\Sh,x)$ be a germ map with $x\in S_r$ and $X=F^{-1}(S_r)$. Suppose $F$ is transverse to the stratum containing $x$. Then, 
$$
F^*(JM(S_r))=JM(X).
$$ 
\end{lemma}
\begin{proof}
It follows directly from lemma $2.10$ of \cite{EAE} and the fact that 
$$
\bigcup_{i=0}^r  \{_iX\backslash _{i-1}X\}
$$ 
is a locally trivial stratification for $X$.
\end{proof}  

\begin{proposition}\label{stable} Let $F:\C^q\longrightarrow \Sh$, $X=F^{-1}(S_r)$, $F$ transverse to $S_r$ at the origin. Then,
$$
JM(X)=F^*(N(S_r))=N(X).
$$
\end{proposition}
\begin{proof} Since $S_r$ is stable we know that $JM(S_r)=N(S_r)$ and, by the previous lemma, $F^*(JM(S_r))=JM(X)$. Since $N$ is universal,
$$
JM(X)=F^*(JM(S_r))=F^*(N(S_r))=N(X).
$$
\end{proof}

\begin{remark} In the last results we have taken the origin only for convenience. All of them are true for any other $x\in X$.
\end{remark}

As a consequence of the last result we have that the multiplicity of the pair is well-defined.

\begin{corollary} Suppose $(X,x)$ is a SEIDS, then $e(JM(X),N(X))$ is well-defined.
\end{corollary}
\begin{proof} Since $X$ is a SEIDS, $F$ is transverse to all the strata of $S_r$ in a punctured neighborhood of the origin. Hence, by the last proposition, on a deleted neighborhood of the origin, $JM(X)=N(X)$. Since the modules agree on such deleted neighborhood, the multiplicity of the pair is well defined. 
\end{proof}

\section{Equisingularity and SEIDS}
 In \cite{EAE}, Gaffney and Ruas extended the framework of \cite{Gaff1} to the simplest non-smoothable and non-isolated singularities - the EIDS. The purpose of this section is to prove similar results for the symmetric case and, in some cases, to give a more explicit formula for the calculations in \cite{EAE}. Let us start by defining some of the invariants we will use.
 
Let us consider $X$ a symmetric determinantal variety given by $X=F^{-1}(S_r)$, and $\X$ its stabilization.  The module $JM_z(\X)$ is generated by the partial derivatives with respect to $z_1,\ldots,z_n$. The first invariant we will use is the $m_d(X)$, where $d$ is the dimension of $X$. This invariant is given by the multiplicity of the polar curve of $JM_z(\X)$ over the parameter space at the origin. In the SEIDS context, this multiplicity is the number of critical points  that a generic linear form has on the complement of the singular set on a generic fiber.

The next invariant is $F(\C^q)\cdot\Gamma_d(S_r)$, which is an intersection number. If we consider the graph of $F$ in $\C^q\times\Sh$, then $F(\C^q)\cdot\Gamma_d(S_r)$ is the intersection of the graph with $\C^q\times\Gamma_d(S_r)$. The next result follows from \cite{Gaff1} and allows us to connect this intersection number with the multiplicity calculated in the last chapter.

\begin{proposition} Suppose $S$ is a stable singularity, and $F:\C^q\longrightarrow\C^p$ is a map defining $X^d=F^{-1}(S)$ with expected codimension. Then, 
$$
\Gamma_i(N(X))=F^{-1}(\Gamma_i(S))
$$
for all $i<d$.
\end{proposition} 

If $\X$ is an one parameter stabilization, we can take the map $\widetilde{F}$ from $\C\times\C^q$ to $\Sh$, so that it is tranverse to $\Gamma_i(S_r)$. Then $F(\C^q)\cdot\Gamma_d(S_r)$ is the number of points in which $\widetilde{F}(y)$ intersects $\Gamma_d(S_r)$. This is the same as the multiplicity over $\C$ of $F^{-1}(\Gamma_d(S_r))$. Using the last proposition and the fact that $\X$ has dimension $d+1$, we conclude that $F(\C^q)\cdot\Gamma_d(S_r)$ is the same as $\mult_{\C}\Gamma_d(N(\X))$.

\begin{proposition} Suppose $\X$ is an one parameter stabilization of a SEIDS $X$ of dimension $d$. Then,
$$
e(JM(X),N(X))+F(\C^q)\cdot\Gamma_d(S_r)=m_d(X).
$$
\end{proposition}
\begin{proof} As we have stated, $F(\C^q)\cdot\Gamma_d(S_r)=\mult_{\C}\Gamma_d(N(\X))$. Now apply the multiplicity polar theorem to the pair $JM_z(\X)$ and $N(\X)$ to get:
$$
\begin{array}{rcl}
\Delta e(JM_z(\X),N(\X))&=&\mult_{\C}\Gamma_d(JM_z(\X))-\mult_{\C}\Gamma_d(N(\X))\\
                            &=&m_d(X)-F(\C^q)\cdot\Gamma_d(S_r).
\end{array}
$$
Now, we only need to prove that $\Delta e(JM_z(\X),N(\X))=e(JM(X),N(X))$. By definition we know that
$$
\Delta e(JM_z(\X),N(\X)=e(JM(\X),N(\X),0)-e(JM(\X),N(\X),y).
$$
Since $\X$ is a SEIDS, by proposition \ref{stable} we have that $JM(\X(y))=N(\X(y))$, which means $e(JM(\X),N(\X),y)=0$. Therefore,
$$
\Delta e(JM_z(\X),N(\X))=e(JM(\X),N(\X),0)=e(JM(X),N(X)).
$$
\end{proof}

For the case where $X=F^{-1}(S_{n-2})$ we can give a more explicit formula to calculate the invariant $m_d(X)$. This is due to the calculations of the mixed polars in the last chapter. This result will be given in the corollary.

\begin{corollary} Suppose $\X$ is family of SEIDS, with $\dim X=d$. Then,
\begin{enumerate}
\item If $n\leq d+2$,
$$
m_d(X)= e(JM(X),N(X)) + \frac{1}{2}\sum_{i=0}^{d+2}\left(\begin{matrix}
									d+2\\
									i		
								\end{matrix}\right)\left(\sum_{l=0}^{n-i}(-1)^{l+1}\colength I_{A_l}\right);
$$
\item If $n>d+2$ and $d-i$ is odd, 
$$
m_d(X)=
e(JM(X),N(X)) + \frac{1}{2}\sum_{i=0}^{d+2}\left(\begin{matrix}
									d+2\\
									i		
								\end{matrix}\right)\left(\sum_{l=0}^{d+3-i}(-1)^{l+1}\colength I_{A_l}\right);
$$.
\item If $n>d+2$ and $d-i$ is even, 
\begin{eqnarray*}
\lefteqn{m_d(X)=e(JM(X),N(X))+}\\
& &+\frac{1}{2}\sum_{i=0}^{d+2}\left(\begin{matrix}
									d+2\\
									i		
								\end{matrix}\right)\left[\left(\sum_{l=0}^{d+3-i}(-1)^{l+1}\colength I_{A_l}\right)-\colength I_{A_{d+3-i}}\right].
\end{eqnarray*}
\end{enumerate}
\end{corollary} 
\begin{proof}
Follows from the last proposition, from the fact that 
$$
\mult_{\C}\Gamma_d(N(\X))=\frac{1}{2}\sum_{i=0}^{d+2}\left(\begin{matrix}
									d+2\\
									i		
								\end{matrix}\right)\deg_{\C}\Gamma_{i,d+2-i}(N(\X)),
$$ 
and from the calculations of the degrees of the mixed polars in \cite{P1}.
\end{proof}

Now, our final goal is to connect the previous results with Whitney equisingularity. A Family $\X$ is Whitney equisingular if given a Whitney stratification of $\X$, each stratum satisfies the Whitney conditions over the parameter $Y$. Consider $\X$ a stabilization of $X$ given by the map
$$
\begin{array}{lclc}
\widetilde{F}: & \mathbb{C}^q\times Y & \longrightarrow & \Sh\\
					& (z,y) & \longmapsto & \widetilde{F}(z,y)
\end{array}
$$
and the modules $JM_z(\X)=\left(\frac{\partial\widetilde{F}}{\partial z_1},\ldots,\frac{\partial\widetilde{F}}{\partial z_q}\right)$ and $m_Y=(z_1,\ldots,z_q)$. For each $y\in Y$ and $i=1,\ldots,r$, consider $_i\X(y)$ the fiber of $_i\X$ at $y$, $d(i)$ the dimension of the symmetric determinantal variety $_iX$ and $\widetilde{f}_y$ the map given by fixing $y$ in $\widetilde{F}(z,y)$. The next result analyzes the relation between the invariant
$$
e(m_YJM(_i\X(y)),N(_i\X(y)))+\widetilde{f}_y(\C^q)\cdot\Gamma_{d(i)}(S_i)
$$
and the condition that the family is Whitney equisingular.

\begin{theorem}\label{goodfamily} Suppose $\X$ is a good $k$-dimensional family of SEIDS. Then the family is Whitney equisingular if and only if the invariants
$$
e(m_YJM(_i\X(y)),N(_i\X(y)))+\widetilde{f}_y(\C^q)\cdot\Gamma_{d(i)}(S_i)
$$
are independent of $y$, where, if $\dim _i\X(y)=0$, by convention the multiplicity of the pair $e(m_YJM(_i\X(y)),N(_i\X(y)))$ is $0$. 
\end{theorem}
\begin{proof}
We can embed the family into a larger one with base $Y\times \C$, so that $y\times\C$ is a stabilization of $\X_y$.
 
Let $\{(_i\X\backslash_{i-1}\X)-Y\}$ be the canonical stratification of $\X-Y$. Since the family is good, by proposition \ref{good} we only need to control the pairs of strata $\left((_i\X - Y),Y\right)$. For each $i$, suppose that the pair $\left((_i\X - Y),Y\right)$ satisfies the Whitney conditions. By theorem \ref{WC} we know that $JM_Y(_i\X)\subset\overline{m_YJM_z(_i\X)}$, which implies that $\overline{m_YJM(_i\X)}=\overline{m_YJM_z(_i\X)}$. In \cite{MPSPCW}, Tessier proved that Whitney conditions implies that the fiber of the blow-up of the conormal modification by $m_Y$ is equidimensional over $Y$, which means that the dimension of the fiber is as small as possible. The integral closure condition $\overline{m_YJM(_i\X)}=\overline{m_YJM_z(_i\X)}$ implies that the same is true for $\B_{m_Y}(\Projan\R(JM_z(_i\X)))$. Therefore, the polar variety of codimension $d(i)$ of $m_YJM_z(_i\X)$ is empty. This implies that the value of $m_{d(i)}(_i\X_y)$ is independent of $y$.
 
Now assume the invariants $m_{d(i)}(_i\X_y)$ are constant.  By the Multiplicity Polar theorem, we have
$$
\begin{matrix}
\Delta e(m_YJM_z(_i\X),N(_i\X))=\mult_Y\Gamma_{d(i)}(m_YJM_z(_i\X))-\mult_Y\Gamma_{d(i)}(N(_i\X))\\
\Downarrow\\
\begin{array}{l}
e(m_YJM(_i\X(0)),N(_i\X(0)))-e(m_YJM(_i\X(y)),N(_i\X(y)))= \\
=\mult_Y\Gamma_{d(i)}(m_YJM_z(_i\X))-\widetilde{f}_0(\C^q)\cdot\Gamma_{d(i)}(S_i)+\widetilde{f}_y(\C^q)\cdot\Gamma_{d(i)}(S_i)\\
\end{array}\\
\Downarrow\\
\begin{array}{c}
\left(e(m_YJM(_i\X(0)),N(_i\X(0)))+\widetilde{f}_0(\C^q)\cdot\Gamma_{d(i)}(S_i)\right)\\
-\left(e(m_YJM(_i\X(y)),N(_i\X(y)))+\widetilde{f}_y(\C^q)\cdot\Gamma_{d(i)}(S_i)\right)=\\
=\mult_Y(\Gamma_{d(i)}(m_YJM_z(_i\X))\\
\end{array}\\
\Downarrow\\
0=\mult_Y(\Gamma_{d(i)}(m_YJM_z(_i\X)).
\end{matrix}
$$
 
Hence the polar $\Gamma_{d(i)}(m_YJM_z(_i\X))$ is empty.
 
Since all the results presented here, from theorem \ref{WC} to Tessier's paper \cite{MPSPCW}, are if and only if results we are done. It is important to notice that when $\dim _i\X=0$ the invariant in question is reduced to $\widetilde{f}_y(\C^q)\cdot\Gamma_0(S_i)$, which is the colength of the defining ideals.
\end{proof}
For the next result, let us consider the following set-up. Consider the symmetric determinantal variety determined by $X=F^{-1}(S_{n-2})$, where $F$ is the map from $\C^q$ to $\Sh$. Consider $q\leq 6$ so that $(X,0)$ has an isolated singularity. 

\begin{theorem} Let $\X$ be a family of isolated singularities as in the given set-up. Then,
\begin{enumerate}
\item If $X$ has a smoothing, then $\X$ is Whitney equisingular if and only if the invariant
$$
e(m_YJM(\X(y)),N(\X(y)))+\frac{1}{2}\sum_{i=0}^{d+2}\left(\begin{matrix}
									d+2\\
									i		
								\end{matrix}\right)\deg_{\C}\Gamma_{i,d+2-i}(N(\X))
$$
is independent of $y$. 
\item If $X$ does not have a smoothing, then $\X$ is Whitney equisingular if and only if the invariant from before and the colength of the ideal defining $_{n-3}\X$, that is, $\widetilde{f}_y(\C^q)\cdot\Gamma_0(S_{n-3})$ are independent of $y$.  
\end{enumerate}
\end{theorem}
\begin{proof}
\begin{enumerate}
\item Since $X$ has a smoothing, all the fibers $\X(y)$ with $y\neq 0$ are smooth. This means that the only pair of stratum in this case is $(\X-Y,Y)$. Proceeding as in the last theorem and using the fact that for all $y\in Y$
$$
\widetilde{f}_y(\C^q)\cdot\Gamma_{d}(S_{n-2})=\frac{1}{2}\sum_{i=0}^{d+2}\left(\begin{matrix}
									d+2\\
									i		
								\end{matrix}\right)\deg_{\C}\Gamma_{i,d+2-i}(N(\X))
$$
we have that $\X$ is Whitney equisingular if and only if the given invariant is constant.
\item Since $X$ does not have a smoothing we have $q=6$ and\linebreak $\dim\widetilde{f}_y^{-1}(S_{n-3})=0$. In this case, we have two pairs of strata to control: $(\X-Y,Y)$ and $(_{n-3}\X-Y,Y)$. The last pair is a dimension zero type of stratum and is controlled only by $\widetilde{f}_y(\C^q)\cdot\Gamma_0(S_{n-3})$. The result follows from the first part of the theorem and the last theorem. 
\end{enumerate}
\end{proof}

For the case where the one parameter stabilization $\X=\widetilde{F}^{-1}(S_{n-2})$ is not a family of isolated singularities we can still state something about the maximal open strata.

\begin{theorem} Let $\X$ be a family of SEIDS, and $\X_0$ the set of its smooth points. Then, the pair $(\X_0-Y,Y)$ satisfies the Whitney conditions if and only if the invariant
$$
e(m_YJM(\X(y)),N(\X(y)))+\frac{1}{2}\sum_{i=0}^{d+2}\left(\begin{matrix}
									d+2\\
									i		
								\end{matrix}\right)\deg_{\C}\Gamma_{i,d+2-i}(N(\X))
$$
is independent of $y$.
\end{theorem}
\begin{proof} The proof is similar to theorem \ref{goodfamily} to the maximal stratum, and follows from the fact that for all $y\in Y$
$$
\widetilde{f}_y(\C^q)\cdot\Gamma_{d}(S_{n-2})=\frac{1}{2}\sum_{i=0}^{d+2}\left(\begin{matrix}
									d+2\\
									i		
								\end{matrix}\right)\deg_{\C}\Gamma_{i,d+2-i}(N(\X)).
$$ 
\end{proof}


\begin{thebibliography}{99}
\bibitem{P1} T. \ Gaffney and M.\ Molino,\emph{Symmetric Determinantal Singularities I: The Multiplicity of the Polar Curve}, \texttt{arXiv:2003.12543 [math.AG]}. 
\bibitem{SDSWE} M. Molino, \emph{Symmetric Determinantal Singularities and Whitney Equisingularity}, Doctoral Dissertation, Universidade Federal Fluminense 2018.
\bibitem{Gaff1} T.\ Gaffney and A.\ Rangachev, \emph{Pairs of Modules and Determinantal Isolated Singularities}, \texttt{arXiv:math/1501.00201v2 [math.CV]}. 
\bibitem{CAV} H.\ Whitney, \emph{Complex Analytic Varieties}, Addison Wesley (1972).
\bibitem{EGZ} W.\ Ebeling and S.\ M.\ Gusein-Zade, \emph{On indices of 1-forms on determinantal singularities}, \texttt{arXiv:math/0806.0219v1 [math.AG]}.
\bibitem{ITEH} D. Eisenbud and J. Harris, \emph{Intersection Theory in Algebraic Geometry}
\bibitem{BRM} D. A. Buchsbaum and D. S. Rim, \emph{A generalized Koszul complex. II. Depth and multiplicity}, Trans. AMS 111 1963 197-224 
\bibitem{TT} A. Rangachev, \emph{Local Volumes, Integral Closures and Equisingularity}, Doctoral Dissertation, Northeastern University 2017.
\bibitem{mather} J.\ Mather, \emph{Notes on Topological Stability}, Harvard University, July 1970.
\bibitem{ICWE} T. Gaffney, \emph{Integral Closure of Modules and Whitney Equisingularity}, Inventiones, 107 (1992) 301-22.
\bibitem{SNTG} T. Gaffney, \emph{Notes in Equisingularity and the Theory of Integral Closure}, Brazil-Mexico 2nd Meeting on Singularities,(2015).
\bibitem{EAE} T. Gaffney and M. A. S Ruas, \emph{Equisingularity and EIDS}, \texttt{arXiv:math/1602.00362v1 [math.CV]} 
\bibitem{MPSPCW} B. Teissier, \emph{Multiplicites polaires, sections planes, et conditions de Whitney}, Springer Lecture Notes in Math. 961 1982 314-491.
\bibitem{GTBRM} S. Kleiman and A. Thorup, \emph{A geometric Theory of the Buchsbaum-Rim Multiplicity}, J. Algebra 167 (1994), 168-231.
\bibitem{BDDMF} M. Zach and J. Mainz, \emph{Bouquet Decomposition for Determinantal Milnor Fibers}, \texttt{arXiv:math/1804.02220v1 [math.AG]}.  
\bibitem{LPVGSS} A. Flores and B. Tessier, \emph{Local Polar Varieties  in the Geometric Study of Singularities}, \texttt{arXiv:math/1607.07979v1 [math.AG]}.
\bibitem{TAV} H. Whitney, \emph{Tangents to an Analytic Variety}, Ann. of Math. (2) 81, 496-549, (1965).
\bibitem{LPAV} H.Whitney, \emph{Local Properties of Analytic Varieties} Differential and Combinatorial Topology (edited by Stewart S. Cairns), (A Symposium in Honor of Marston Morse), Princeton Univ. Press, Princeton, NJ, 205-244,(1965).
\bibitem{ICMWE} T. Gaffney, \emph{Integral Closure of Modules and Whitney equisingularity}, Invent. Math. 107, 301-322,(1992). 
\bibitem{FSM} J. W. Bruce, \emph{Families of symmetric matrices}, Moscow Math. J., 3, no 2, 335-360, (2003).
\bibitem {SIDM} T. Gaffney and S. Kleiman, \emph{Specialization of integral dependence for modules}, Inventiones mathematicae.137,541-574, (1999).
\bibitem{SPSM} C. G. Gibson, \emph{Singular points of smooth mappings}, Research Notes in Mathematics 25, (1979).
\bibitem{MPT} T. Gaffney, \emph{The Multiplicity Polar Theorem}, arxiv:math/0703650v1 [math.CV], (2007).
\bibitem{SNHS} T. Gaffney and R. Gassler, \emph{Segre numbers and hypersurface singularities}, J. Algebraic Geometry 8,695-736,(1999).
\bibitem{CESPCW} D. L\^e and B. Teissier, \emph{Cycles evanescents, sections planes et conditions de Whitney}, II. (French) [Vanishing cycles, plane sections and Whitney conditions. II] Singularities, Part 2 (Arcata, Calif., 1981) Proc. Symp. Pure. Math., 40, Amer. Math. Soc., Providence, RI, 65-103,(1983).
\bibitem{MEIG} T. Gaffney, \emph{Multiplicity and equisingularity of ICIS germs}. Invent. Math., 123(2):209-220, (1996).
\bibitem{PMIPME} T. Gaffney, \emph{Polar methods, invariants of pairs of modules and equisingularity}. In Real and complex singularities, volume 354 of Comtemp. Math., pages 113-135. Amer. Math. Soc., Providence, RI, (2004).
\bibitem{ICIRM} I. Swanson and C. Huneke, \emph{Integral Closure of Ideals, Rings and Modules}, Cambridge University Press, Cambridge, 2006.
\bibitem{MPTIS} T. Gaffney, \emph{The Multiplicity polar theorem and isolated singularities}, J. Algebraic Geom. 18, no. 3, 547-574, (2009).
\bibitem{MPFE} T. Gaffney, \emph{The Multiplicity-Polar Formula and Equisingularity}, in preparation.
\bibitem{TFBRM} S. Kleiman, \emph{Two formulas for the BR multiplicity}, Ann Univ Ferrara, DOI 10.1007/s11565-016- 0250-2.
\bibitem{LT} M. Lejeune-Jalabert and B. Teissier, \emph{Cloture int\' egrale des id\' eaux et equisingularit\' e, avec 7 compl\' ements}. Annales de la Facult\' e des Sciences de Toulouse, Vol XVII, No.4, 781-859, (2008).
\end{thebibliography}
 \end{document}